\documentclass[3p,times]{elsarticle}

\usepackage{amsmath}
\usepackage{amsthm}
\usepackage{booktabs}
\newtheorem{prop}{Proposition}

\newtheorem{defn}{Definition}

\newtheorem{propert}{Properties}
\newtheorem{cor}{Corollary}
\newtheorem{exmp}{Example}
\newtheorem{obs}{Observation}
\newtheorem{rem}{Remark}

\usepackage{amssymb}
\usepackage[figuresright]{rotating}
\usepackage{tcolorbox}

\usepackage[font=footnotesize,labelfont=bf]{caption}
\usepackage[font=footnotesize,labelfont=bf]{subcaption}

\usepackage[english]{babel}
\usepackage{empheq,xcolor}

\begin{document}
	\begin{frontmatter}
		\title{Umbral-Algebraic Methods and Asymptotic Properties of Special Polynomials}

\author[Enea]{G. Dattoli}
\ead{giuseppe.dattoli@enea.it}

\author[Enea]{S. Licciardi\corref{cor}}
\ead{silvia.licciardi@enea.it}

\author[Unict]{R.M. Pidatella}
\ead{rosa@dmi.unict.it}

\author[Enea]{E. Sabia}
\ead{elio.sabia@enea.it}

\address[Enea]{ENEA - Frascati Research Center, Via Enrico Fermi 45, 00044, Frascati, Rome, Italy}
\cortext[cor]{Corresponding author: silviakant@gmail.com, silvia.licciardi@enea.it, orcid 0000-0003-4564-8866, tel. nr: +39 06 94005421. }
\address[Unict]{Dep. of Mathematics and Computer Science, University of Catania, Viale A. Doria 6, 95125, Catania, Italy}

\begin{abstract}
A new method of algebraic nature is proposed for the study of the asymptotic properties of special polynomials. The technique we foresee is based on the use of umbral operators, allowing a unified treatment of a large body of polynomial families, with the use of elementary algebraic tools.
\end{abstract}


\end{frontmatter}

\section{Introduction}	
	
%
%
%
%
%
%
%
%
%
%
%
%

The Laguerre polynomials in umbral form are defined as \cite{Babusci,BDGPe,lacunary}

\begin{equation}\label{eq:Lag pol umbral}
L_{n}(x,y)=\left(y-\hat{c}x\right)^{n}\varphi_{0}\;,
\end{equation}
where we have introduced an umbral operator $\hat{c}$ and a vacuum
$\varphi_{0}$ \cite{S.Roman,SLicciardi,On Ramanujan}. Albeit they can be defined on abstract grounds, we
can provide the relevant realization in terms of a differential operator and special functions, as specified below.

\begin{defn}
	Let $\varphi(z)$ an appropriate chosen function
	 \begin{equation}\label{key}
	 \varphi\left(z\right):=\varphi_{z}=\frac{1}{\Gamma\left(z+1\right)}.
	 \end{equation}
	We introduce the Umbral operator $\hat{c}$ as a shift operator 
	\begin{equation}\label{key}
	\hat{c} =e^{\partial_{z}}
	\end{equation}
	acting on the space of the functions, called vacuum $\varphi_{0}$, according to the prescription \cite{SLicciardi,L.C.Andrews}
	in the following way
\begin{equation}
\hat{c}^{\,\mu}\varphi_{0}=e^{\,\mu\partial_{z}}\frac{1}{\Gamma\left(z+1\right)}\mid_{z=0}=\frac{1}{\Gamma\left(\mu+1\right)}.\label{eq:operator c}
\end{equation}
\end{defn}
\noindent It is furthermore evident that the umbral operator $\hat{c}$ satisfies the following properties.

\begin{propert}\label{eq:oper prop}
	$\forall \mu,\nu\in\mathbb{R}$
	
\begin{align}
\hat{c}^{\,\mu}\hat{c}^{\,\nu}=\hat{c}^{\,\mu+\nu}\;,\\
\left(\hat{c}^{\,\mu}\right)^{r}=\hat{c}^{\,r\mu}\;.
\end{align}
\end{propert}

\begin{obs}
By expanding the Newton binomial in eq. \eqref{eq:Lag pol umbral}
and in view of the summarized rules, we find

\begin{equation}\label{eq:Newton expans}
L_{n}\left(x,y\right)=\sum_{s=0}^n\binom{n}{s}(-1)^s y^{n-s}\;\hat{c}^s\;x^s\;\varphi_{0}=
\sum_{s=0}^n\dfrac{1}{s!}\binom{n}{s}(-1)^s y^{n-s}x^s
\end{equation}
which yields the usual definition \cite{L.C.Andrews}.
\end{obs}

The Umbral image of Bessel functions has been shown to be a Gaussian, namely \cite{SLicciardi}-\cite{D.Babusci}

\begin{equation}\label{eq:umbral Bessel}
J_{n}\left(x\right)=\left(\hat{c}\frac{x}{2}\right)^{n}e^{-\hat{c}\left(\frac{x}{2}\right)^{2}}\varphi_{0}=\sum_{r=0}^{\infty}\dfrac{\left(-1\right)^{r}\left(\frac{x}{2}\right)^{n+2r}}{r!\left(n+r\right)!}\;.
\end{equation}
The previous restyling has been profitably exploited to study most of the properties of Laguerre polynomials and of special functions as well, by using very straightforward algebraic means. We will hereafter show that the same formalism allows the derivation of the asymptotic properties of Laguerre families using the same elementary tools.

\begin{exmp}
 To this aim we note that, according to eq. (\ref{eq:Lag pol umbral}),
the following identity holds

\begin{equation}\label{eq:Laguerre approx Bessel}
\left.  L_{n}\left(\frac{x}{n},y\right)\right|  _{n\gg 1}=\left. y^{n}\left(1-\hat{c}\frac{x}{yn}\right)^{n}\right| _{n\gg 1}\,\varphi_{0}\simeq y^{n}e^{-\hat{c}\frac{x}{y}}\varphi_{0}=y^{n}J_{0}\left(2\sqrt{\frac{x}{y}}\right)\;,
\end{equation}
which is a well known result \cite{Lebedev}, obtained in a fairly direct
way. 
\end{exmp}

\noindent A better approximation can be obtained according to the Proposition below.

\begin{prop}
	$\forall x,y\in\mathbb{R}$, $\forall n\in\mathbb{N}:n\gg1$
	\begin{equation}\label{eq:Laguerre as Herm based Bess}
	\left. L_{n}\left(\frac{x}{n},y\right)\right| _{n\gg 1}\simeq
	y^{n}{}_{H}C_0\left(-\frac{x}{y},-\frac{1}{2n}\left(\frac{x}{y}\right)^{2}\right)
	\end{equation}
where ${}_{H}C_n\left(x,y\right)$ are the Hermite based Bessel function \cite{DattoliComp}
\begin{equation}
{}_{H}C_n\left(x,y\right)=\sum_{r=0}^{\infty}\frac{H_{r}\left(x,y\right)}{r!\left(n+r\right)!}\label{eq:Hermite based Bessel func}\;.
\end{equation}
		\end{prop}
\begin{proof}
	We note that
\begin{equation}\label{eq:identity n.1}
\left(y-\hat{c}\frac{x}{n}\right)^{n}\varphi_{0}=y^{n}e^{n\ln\left(1-\frac{\hat{c}x}{ny}\right)}\varphi_{0}\;.
\end{equation}
By expanding the logarithm up to the second order, we find

\begin{equation}\label{eq:log expan}
\left. L_{n}\left(\frac{x}{n},y\right)\right| _{n\gg1}\simeq y^{n}e^{-\hat{c}\frac{x}{y}-\hat{c}^{2}\frac{1}{2n}\left(\frac{x}{y}\right)^{2}}\varphi_{0}\;.
\end{equation}
The use of the two variable Hermite polynomials \cite{Appel}

\begin{equation}\label{eq:Two var Hermite}
H_{n}\left(x,y\right)=n!\sum_{r=0}^{\lfloor\frac{n}{2}\rfloor}\frac{x^{n-2r}y^{r}}{\left(n-2r\right)!r!}
\end{equation}
with  generating function \cite{Babusci} 

\begin{equation}\label{key}
\sum_{n=0}^{\infty}\frac{t^{n}}{n!}H_{n}\left(x,y\right)=e^{xy+yt^{2}}
\end{equation}
yields the result

\begin{equation}
\left. L_{n}\left(\frac{x}{n},y\right)\right| _{n\gg 1}\simeq y^{n}\sum_{r=0}^{\infty}\frac{\hat{c}^{\,r}}{r!}H_{r}\left(-\frac{x}{y},-\frac{1}{2n}\left(\frac{x}{y}\right)^{2}\right)\varphi_{0}
=y^{n}{}_{H}C_0\left(-\frac{x}{y},-\frac{1}{2n}\left(\frac{x}{y}\right)^{2}\right).
\end{equation}
%
\end{proof}

The level of approximation provided by Eq. \eqref{eq:Laguerre as Herm based Bess} is significantly better than its conventional counterpart given in Eq. \eqref{eq:Laguerre approx Bessel}. We have numerically checked the advantages of the new approximation and we have summarized an example in Tab. \ref{tab1} for  $x=y=1$ and $n=10$.

\begin{table}[h]\caption{\textbf{Level of approximation}}\label{tab1}
	\centering
	{\renewcommand\arraystretch{1.7} 
		\begin{tabular}{||p{0.9in}|p{0.9in}|p{0.9in}||}
			\hline   
			 Reference value &  Approximation & Relative error  \\ \hline
			\toprule
		\hspace{.5cm}  Exact value & \hspace{.3cm} $0.2058543$ &\hspace{0.9cm} $-$  \\ \hline
		 \hspace{.5cm} Eq. \eqref{eq:Laguerre approx Bessel} &\hspace{.3cm}  $0.2238908$ &\hspace{.5cm} $8.7\cdot 10^{-2}$ \\ \hline
			\hspace{.5cm} Eq. \eqref{eq:Laguerre as Herm based Bess} &\hspace{.3cm}  $0.2062915$ &\hspace{.5cm} $2.1\cdot 10^{-3}$ \\ \hline
			\bottomrule
	\end{tabular}}
\end{table}
%
%
 
 In the following we will derive even better results and further comment on this aspect of the problem in the concluding
section.\\

We have so far fixed the formalism we will follow in the forthcoming
parts of the paper, in which we will deal with the large index polynomial
expansion of different families of polynomials.

\section{Associated Laguerre and Higher Order Hermite Polynomials}

The main result of the previous section is provided by Eq. \eqref{eq:Laguerre as Herm based Bess} which can be further improved. 

\begin{rem}
For the sake of consistency we note
that the fact that we have considered the expansion at the second order in the argument of the logarithm allows to expand the exponential, in such a way that 

\begin{equation}
\left. L_{n}\left(\frac{x}{n},y\right)\right| _{n\gg 1}\simeq y^{n}e^{-\hat{c}\frac{x}{y}-\hat{c}^{2}\frac{1}{2n}\left(\frac{x}{y}\right)^{2}}\varphi_{0}\simeq y^{n}\left(1-\frac{1}{2n}\left(\frac{x}{y}\right)^{2}\hat{c}^{\,2}\right)e^{-\hat{c}\frac{x}{y}}\varphi_{0}\label{eq:exponential expans}
\end{equation}
which, on account of the previous identities, yields 

\begin{equation}\label{eq:new express}
\left. L_{n}\left(\frac{x}{n},y\right)\right| _{n\gg 1}\simeq y^{n}\left(J_{0}\left(2\sqrt{\frac{x}{y}}\right)-\frac{1}{2n}\left(\frac{x}{y}\right)J_{2}\left(2\sqrt{\frac{x}{y}}\right)\right)\;,
\end{equation}
providing the same order of accuracy of the previous expression involving an Hermite based Bessel.
\end{rem}

As already stressed, we can obtain even better approximations if we further expand the logarithm in Eq. \eqref{eq:identity n.1} thus finding, e.g., an other representation for Eq. \eqref{eq:Laguerre approx Bessel} as it shown below.

\begin{cor}
	$\forall x,y\in\mathbb{R}$, $\forall n\in\mathbb{N}:n\gg1$

\begin{equation}\label{eq:explicit better expans}
\left. L_{n}\left(\frac{x}{n},y\right)\right| _{n\gg 1}\simeq y^{n}{}_{H}C_{0}\left(-\left\{ \frac{1}{s\;n^{s-1}}\left(\frac{x}{y}\right)^{s}\right\} _{s=1}^{m}\right)\;.
\end{equation}
\end{cor}
\begin{proof}
	Let
\begin{equation}\label{eq:better expans}
\left. L_{n}\left(\frac{x}{n},y\right)\right| _{n\gg 1}\simeq y^{n}e^{-\sum_{s=1}^m\frac{\hat{c}^{\,s}}{s\;n^{s-1}}\left(\frac{x}{y}\right)^{s}}\varphi_{0}\;.
\end{equation}
The use of the $m$-variable Hermite polynomials, defined as 

\begin{equation}
\label{eq:m-var Hermite poly}
 H_{n}^{(m)}\left(\left\{ x\right\} _{1}^{m}\right)=n!\sum_{r=0}^{\lfloor\frac{n}{m}\rfloor}\frac{x_{m}^{r}H_{n-mr}^{\left(m-1\right)}\left(\left\{ x\right\} _{1}^{m-1}\right)}{\left(n-mr\right)!r!}\;,\qquad \left\{ x\right\} _{1}^{m}=x_{1},\,x_{2},...,\,x_{m}\;,
\end{equation}
with generating function \cite{DattoliComp}

\begin{equation}\label{eq:m-vr Hermite poly generat func}
\sum_{n=0}^\infty\frac{t^{n}}{n!}H_{n}^{\left(m\right)}\left(\left\{ x\right\} _{1}^{m}\right)=e^{\sum_{s=1}^m x_{s}t^{s}}\;,
\end{equation}
provides the Eq. \eqref{eq:explicit better expans}
%
\end{proof}
To give an idea of the level of approximation we note that, for $n=5,x=y=1$, we have Tab. \ref{tab2}

\begin{table}[h]\caption{\textbf{Level of approximation}}\label{tab2}
	\centering
	{\renewcommand\arraystretch{1.7} 
		\begin{tabular}{||p{0.9in}|p{0.9in}|p{0.9in}||}
			\hline   
			Reference value &  Approximation & Relative error  \\ \hline
			\toprule
			  Exact value & \hspace{.3cm} $0.1869973$ &\hspace{0.9cm} $-$  \\ \hline
			 Eq. \eqref{eq:Laguerre approx Bessel} &\hspace{.3cm}  $0.2238908$ &\hspace{.5cm} $1.9\cdot 10^{-1}$ \\ \hline
			 Eqs. \eqref{eq:Laguerre as Herm based Bess}, \eqref{eq:new express} &\hspace{.3cm}  $0.1887772$ &\hspace{.5cm} $9.5\cdot 10^{-3}$ \\ \hline
 Eq. \eqref{eq:explicit better expans}, $m=6$ &\hspace{.3cm}  $0.1870019$ &\hspace{.5cm} $2.5\cdot 10^{-5}$ \\ \hline
			\bottomrule
	\end{tabular}}
\end{table}
\noindent while, in Tab. \ref{tab3}, we have levels of approximation for $n=3,y=3,x=1$ which yields, even for a not large index, quite a good approximation. 

\begin{table}[h]\caption{\textbf{Level of approximation}}\label{tab3}
	\centering
	{\renewcommand\arraystretch{1.7} 
		\begin{tabular}{||p{0.9in}|p{0.9in}|p{0.9in}||}
			\hline   
			Reference value &  Approximation & Relative error  \\ \hline
			\toprule
			Exact value & \hspace{.3cm} $18.4938272$ &\hspace{0.9cm} $-$  \\ \hline
			Eq. \eqref{eq:Laguerre approx Bessel} &\hspace{.3cm}  $18.7227933$ &\hspace{.5cm} $1.2\cdot 10^{-2}$ \\ \hline
			Eqs. \eqref{eq:Laguerre as Herm based Bess}, \eqref{eq:new express} &\hspace{.3cm}  $18.4996194$ &\hspace{.5cm} $3.1\cdot 10^{-4}$ \\ \hline
			Eq. \eqref{eq:explicit better expans}, $m=5$ &\hspace{.3cm}  $0.184938301$ &\hspace{.5cm} $1.6\cdot 10^{-7}$ \\ \hline
			\bottomrule
	\end{tabular}}
\end{table}

Within the context of the umbral formalism, the Associated Laguerre polynomials \cite{Babusci} are expressed as

\begin{equation}\label{eq:associated Laguerre}
L_{n}^{\left(\alpha\right)}\left(x,y\right)=\frac{\Gamma\left(n+\alpha+1\right)}{n!}\hat{c}^{\,\alpha}\left(y-\hat{c}x\right)^{n}\varphi_{0}\;.
\end{equation}
We obtain therefore the further Corollary \ref{C2}.

\begin{cor}\label{C2}
$\forall x,y,\alpha\in\mathbb{R}$, $\forall n\in\mathbb{N}:n\gg1$
\begin{equation}\label{eq:eval associated Laguerre}
\left.  \frac{n!}{\Gamma\left(n+\alpha+1\right)}L_{n}^{\left(\alpha\right)}\left(x,y\right)\right| _{n\gg 1}\simeq y^{n}\hat{c}^{\,\alpha}e^{-\sum_{s=1}^m \frac{\hat{c}^{\,s}}{s\;n^{s-1}}\left(\frac{x}{y}\right)^{s}}\varphi_{0}=y^{n}{}_{H}C_{\alpha}\left(-\left\{  \frac{1}{sn^{s-1}}\left(\frac{x}{y}\right)^{s}\right\} _{s=1}^{m}\right)
\end{equation}
which at the lowest order $(m=1)$ yields the well-known expression 

\begin{equation}
\left. \frac{n!}{\Gamma\left(n+\alpha+1\right)}L_{n}^{\left(\alpha\right)}\left(x,y\right)\right| _{n\gg 1}\simeq y^{n}C_{\alpha}\left(\frac{x}{y}\right)=y^{n}\left(2\frac{x}{y}\right)^{\alpha}J_{\alpha}\left(2\sqrt{\frac{x}{y}}\right).
\label{eq:at lowest order}
\end{equation}
\end{cor}

In this section we have derived, as unexpected feature, the link between large index Laguerre, higher order Hermite and non-standard multivariable Bessel functions. In the forthcoming section we will show how, the umbral formalism we have sketched so far allows a strictly analogous treatment for the Hermite polynomial theory and provides a closely similar procedures for the derivation of the relevant large index expansion.

\section{Large Index Expansion of Hermite Polynomials}

In the following we will use the two variable Hermite introduced in Eq. \eqref{eq:Two var Hermite} and note that they satisfy the scaling relation \cite{Babusci} 

\begin{equation}\label{eq:scaling rel}
a^{n}H_{n}\left(x,y\right)=H_{n}\left(ax,a^{2}y\right)\;.
\end{equation}
It is accordingly evident that 

\begin{equation}\label{eq:acc. evident}
\frac{y^{\frac{n}{2}}}{n^{n}}H_{n}\left(\frac{nx}{\sqrt{y}},1\right)=H_{n}\left(x,\frac{y}{n^{2}}\right)\;.
\end{equation}
The umbral form of Hermite polynomials is provided by \cite{G.Dattoli,SLicciardi}\footnote{The vacuum and the umbral operator are in this case defined as \begin{equation*}\label{key}
	 \hat{h}_{y}=e^{\sqrt{y}\partial_{\xi}},\qquad  \quad \phi_{0}=\frac{\Gamma\left(\xi+1\right)}{\Gamma\left(\frac{\xi}{2}+1\right)}\left|\cos\left(\xi\frac{\pi}{2}\right)\right|.
\end{equation*}} 

\begin{equation}\label{eq:umbral of Herm poly}
H_{n}\left(x,y\right)= \left(x+\hat{h}_{y}\right)^{n}\phi_{0}\;,
\end{equation}
where the umbral operator $\hat{h}$ acts on the vacuum $\phi_0$ providing
 
\begin{equation}
\hat{h}_{y}^{\,r}\phi_{0}=\frac{y^{\frac{r}{2}}r!}{\Gamma\left(\frac{r}{2}+1\right)}\left|\cos\left(r\frac{\pi}{2}\right)\right|\label{eq:oper definit}\;.
\end{equation}
It is worth noting that \cite{SLicciardi} 

\begin{equation}\label{eq:another identity}
e^{\hat{h}_{y}z}\phi_{0}=\sum_{r=0}^\infty\frac{\left(\hat{h}_{y}z\right)^{r}}{r!}\phi_{0}=e^{yz^{2}}.
\end{equation}

It is now fairly natural to follow the same steps leading to the Laguerre asymptotic forms.

\begin{prop}
 We set, $\forall x,y\in\mathbb{R}$, $\forall n\in\mathbb{N}:n\gg1$

\begin{equation}\label{eq:asymptotic form}
\left. H_{n}\left(x,\frac{y}{n^{2}}\right)\right| _{n\gg 1}=\left. x^{n}\left(1+\frac{\hat{h}_{y}}{xn}\right)^{n}\right| _{n\gg1}\phi_{0}
\simeq x^{n}e^{\frac{\hat{h}_{y}}{x}}\phi_{0}=x^{n}e^{\frac{y}{x^{2}}}\;.
\end{equation}
The expansion at the second order yields

\begin{equation}\label{eq:second order expansion}
\left. H_{n}\left(x,\frac{y}{n^{2}}\right)\right| _{n\gg 1}\simeq
 x^{n}e^{\frac{\hat{h}_{y}}{x}\frac{\hat{h}_{y}^{\,2}}{2x^{2}n}}\phi_{0}=
 x^{n}\sum_{r=0}^\infty\frac{1}{r!}H_{2r}\left(\frac{\sqrt{y}}{x},-\frac{y}{2nx^{2}}\right)\;.
\end{equation}
\end{prop}

The approximation given by Eq. \eqref{eq:second order expansion} is easily checked to be significantly better than the corresponding lower order case, as seen from the numerical examples reported in Tab. \ref{tab4}  for $x=1,y=3,n=70$.

\begin{table}[h]\caption{\textbf{Level of approximation}}\label{tab4}
	\centering
	{\renewcommand\arraystretch{1.7} 
		\begin{tabular}{||p{0.9in}|p{0.9in}|p{0.9in}||}
			\hline   
			Reference value &  Approximation & Relative error  \\ \hline
			\toprule
		\hspace{.3cm} 	Exact value & \hspace{.6cm} $15.465$ &\hspace{0.9cm} $-$  \\ \hline
		\hspace{.3cm} 	Eq. \eqref{eq:asymptotic form} &\hspace{.6cm}  $20.086$ &\hspace{.5cm} $2.3\cdot 10^{-1}$ \\ \hline
		\hspace{.3cm} 	Eq. \eqref{eq:second order expansion} &\hspace{.6cm}  $15.211$ &\hspace{.5cm} $1.6\cdot 10^{-2}$ \\ \hline
			\bottomrule
	\end{tabular}}
\end{table}
The inclusion of higher order corrections follows the same steps as in the case of Laguerre polynomials and indeed, for $x=y=3$, we find for example Tab. \ref{tab5}.\\

\begin{table}[h]\caption{\textbf{Level of approximation}}\label{tab5}
	\centering
	{\renewcommand\arraystretch{1.7} 
		\begin{tabular}{||p{0.9in}|p{0.9in}|p{0.9in}||}
			\hline   
			Reference value &  Approximation & Relative error  \\ \hline
			\toprule
			Exact value &  $7.84727363\cdot 10^4$ &\hspace{0.9cm} $-$  \\ \hline
			Eq. \eqref{eq:asymptotic form} & $7.81475219\cdot 10^4$ &\hspace{.5cm} $4.1\cdot 10^{-3}$ \\ \hline
			Eq. \eqref{eq:second order expansion}, $m=3$ &  $7.85216889\cdot 10^4$ &\hspace{.5cm} $6.2\cdot 10^{-4}$ \\ \hline
			Eq. \eqref{eq:second order expansion}, $m=4$ & $7.84655402\cdot 10^4$ &\hspace{.5cm} $9.1\cdot 10^{-5}$ \\ \hline
			\bottomrule
	\end{tabular}}
\end{table}

In this section we have shown that the umbral technique offers a very efficient tool for the evaluation of the asymptotic series for Hermite polynomials in complete analogy with the Laguerre case.\\

The forthcoming section is devoted to further examples involving mixed polynomial families.

\section{Final Examples and Concluding Comments}

The second order expansion in Eq. \eqref{eq:second order expansion} is amenable for a simplification yielding the second order asymptotic expansion in terms of a Gauss function. \\

We remind indeed the generating function involving even index of Hermite polynomials \cite{Motzkin}

\begin{equation}\label{eq:even index gen func Herm}
\sum_{n=0}^\infty\frac{t^{n}}{n!}H_{2n}\left(x,y\right)=\frac{1}{\sqrt{1-4yt}}e^{\frac{tx^{2}}{1-4yt}}
\end{equation}
which yields for Eq. \eqref{eq:second order expansion} the following identity 

\begin{equation}\label{eq:following identity}
\left. H_{n}\left(x,\frac{y}{n^{2}}\right)\right| _{n\gg 1}\simeq\frac{\sqrt{n}x^{n+1}}{\sqrt{nx^{2}+2y}}\,e^{\frac{ny}{nx^{2}+2y}}\;.
\end{equation}

We close this note by adding a few comments to the extension of the method to other family of polynomials as for example the hybrid Laguerre-Hermite (h-LH), which have been exploited in the past to deal with the so called Motzkin numbers and with their generalization \cite{Blasiak}.\\

The h-LH are polynomials in between Laguerre and Hermite, hence the name, they are indeed defined as

\begin{equation}\label{eq:h-LH poly def}
HL_{n}\left(x,y\right)=n!\sum_{r=0}^{\lfloor\frac{n}{2}\rfloor}\frac{x^{n-2r}y^{r}}{\left(n-2r\right)!r!}=H_{n}\left(x,\,\hat{c}y\right)\varphi_{0}=\left(x+\sqrt{\hat{c}}\;\hat{h}_{y}\right)^{n}\phi_{0}\;\varphi_{0}\;,
\end{equation}
where the operators $\hat{c},\,\hat{h}$ act separately on the vacua $\phi_{0},\,\varphi_{0}$ respectively. \\

\noindent The use of the same procedure as before yields the equation below. 

\begin{prop}
$\forall x,y\in\mathbb{R}$, $\forall n\in\mathbb{N}:n\gg1$
\begin{equation}\label{eq:h LH poly eval}
 \left. HL_{n}\left(x,\frac{y}{n^{2}}\right)\right| _{n\gg1 }\simeq x^{n}e^{\frac{\sqrt{\hat{c}}\hat{h}_{y}}{x}}\phi_{0}\varphi_{0}=x^{n}e^{\hat{c}\frac{y}{x^{2}}}\varphi_{0}=x^{n}I_{0}\left(2\frac{\sqrt{y}}{x}\right),
\end{equation}
where 

\begin{equation}\label{key}
I_{0}\left(x\right)=\sum_{r=0}^\infty\frac{\left(\frac{x}{2}\right)^{n+2r}}{r!\left(n+r\right)!}
\end{equation}
 is the $0$-order modified Bessel function of the first kind.
\end{prop}

\begin{cor}
The higher order approximation leads to expressions umbrally equivalent to those derived in the previous sections thus finding, e.g., 

\begin{equation}\label{eq:HL high order approx}
\left. HL_{n}\left(x,\frac{y}{n^{2}}\right)\right| _{n\gg 1}\simeq x^{n}e^{\frac{\sqrt{\hat{c}}\hat{h}_{y}}{x}-\frac{\hat{c}\hat{h}_{y}^{2}}{2x^{2}n}}\;\phi_{0}=x^{n}\sum_{r=0}^\infty \frac{y^{r}}{\left(r!\right)^{2}}H_{2r}\left(\frac{1}{x},-\frac{1}{2nx^{2}}\right)\;.
\end{equation}
\end{cor}

This note has provided a description of a use of umbral/operational methods to define a strategy for the study of the asymptotic properties of special polynomials.\\

 In a future investigation we will extend the procedure to the case of special functions.\\
 
\textbf{Acknowledgements}\\

The work of Dr. S. Licciardi was supported by an Enea Research Center individual fellowship.\\

\textbf{Author Contributions}\\

Conceptualization: G.D.; methodology: G.D., S.L.; data curation: S.L., R.M.P.; validation: G.D., S.L., E.S., R.M.P.; formal analysis: G.D., S.L.; writing - original draft preparation: G.D., S.L.; writing - review and editing: S.L. .\\


\textbf{References}
 
\begin{thebibliography}{}
	
\bibitem{Babusci} Babusci D., Dattoli G., Licciardi S., Sabia E. Mathematical Methods for Physics,  World Scientific, Singapore, 2019.

\bibitem{BDGPe} D. Babusci, G. Dattoli, K. Gorska, and K.A. Penson, \textit{"Generating functions for Laguerre polynomials: New identities for lacunary series"}, arXiv:1210.3710.

\bibitem{lacunary} Babusci, D.; Dattoli, G.; Gorska, K., Penson, K.A. Lacunary Generating Functions for Laguerre Polynomials, \textit{Sem. Loth. Comb.}, \textbf{2017}, \textit{76}, Article B76b.

\bibitem{S.Roman} Roman, S. The Umbral Calculus, Dover Publications, New York, 2005. 

\bibitem{SLicciardi} Licciardi, S. Umbral Calculus, a Different Mathematical Language [Ph D Thesis].
University of Catania, 2018,  arXiv:1803.03108 [math.CA].
	
\bibitem{On Ramanujan} Babusci, D.; Dattoli, G. On Ramanujan Master Theorem, arXiv:1103.3947 [math-ph].
	
 \bibitem{L.C.Andrews} Andrews, L.C. Special Functions For Engeneers and Applied mathematicians, Mc Millan, New York, 1985.

\bibitem{D.Babusci} Babusci, D.; Dattoli, G.; Gorska, K.; Penson, K.A. The spherical Bessel and Struve functions and operational methods, \textit{Appl. Math. Comput.}, \textbf{2014}, \textit{238}, 1--6.

\bibitem{Lebedev} Lebedev, N.N. Special Functions and Applications, Dover Publication, New York, 1972.

\bibitem{DattoliComp} Dattoli, G. Generalized polynomials, operational identities and their applications, \textit{J. Comput. Appl. Math.}, \textbf{2000}, \textit{118}, 111--123.

\bibitem{Appel} App\'el, P.;  Kamp\'e de F\'eri\'et, J. Fonctions Hypergeometriques and Hyperspheriques. Polynomes d'Hermite, Gauthiers-Villars, Paris, 1926.

\bibitem{G.Dattoli} Dattoli, G.; Germano, B.; Martinelli, M.R.; Ricci, P.E. Lacunary Generating
Functions of Hermite polynomials and Symbolic methods, \textit{Ilirias J. Math.}, \textbf{2015}, \textit{4}, 16--23.

\bibitem{Motzkin} Artioli, M.; Dattoli, G.; Licciardi, S.; Pagnutti, S. Motzkin numbers: an operational point of view, \textit{J. Integer Seq.}, \textbf{2018}, \textit{21}, Article 18.7.5, cited on Online Electronic Integer Sequences as arXiv:1703.07262.

\bibitem{Blasiak} Blasiak, P.; Dattoli, G.; Horzela, A.; Penson, K.A.; Zhukovsky, K. Motzkin numbers, central trinomial coefficients and hybrid polynomials, \textit{J. Integer Seq.}, \textbf{2008}, \textit{11}, Art. 08.1.1.
 

\end{thebibliography}
\end{document}